\documentclass[a4paper]{amsart}

\usepackage{a4wide}
\usepackage[UKenglish]{babel}
\usepackage{amsmath,amsfonts,amsthm}
\usepackage{graphicx}
	\graphicspath{{./}{figures/}}
\usepackage{enumerate}

\theoremstyle{remark}\newtheorem{remark}{Remark}[section]
\theoremstyle{plain}
	\newtheorem{theorem}[remark]{Theorem}
	\newtheorem{proposition}[remark]{Proposition}
	\newtheorem{lemma}[remark]{Lemma}
	\newtheorem{corollary}[remark]{Corollary}
\theoremstyle{definition}
	\newtheorem{definition}[remark]{Definition}

\newcommand{\abs}[1]{\left\vert#1\right\vert}

\newcommand{\G}{\Gamma}
\newcommand{\cE}{\mathcal{E}}

\newcommand{\cM}{\mathcal{M}}
\newcommand{\cS}{\mathcal{S}}
\newcommand{\cT}{\mathcal{T}}
\newcommand{\cW}{\mathcal{W}}
\newcommand{\cV}{\mathcal{V}}
\newcommand{\dual}[2]{\langle#1,\,#2\rangle}

\newcommand{\normBL}[1]{\left\Vert#1\right\Vert^\ast_{BL}}
\newcommand{\R}{{\mathbb R}}
\newcommand{\N}{{\mathbb N}}
\newcommand{\supp}[1]{\operatorname{supp}#1}

\newcommand{\p}{\partial}
\numberwithin{equation}{section}
\begin{document}

\title[]{Measure-valued solutions to nonlocal transport equations on networks}

\author{Fabio Camilli}
\address{Dipartimento di Scienze di Base e Applicate per l'Ingegneria, ``Sapienza'' Universit{\`a}  di Roma, Via Scarpa 16, 00161 Rome, Italy}
\email{camilli@sbai.uniroma1.it}

\author{Raul De Maio}
\address{Dipartimento di Scienze di Base e Applicate per l'Ingegneria, ``Sapienza'' Universit{\`a}  di Roma, Via Scarpa 16, 00161 Rome, Italy}
\email{raul.demaio@sbai.uniroma1.it}

\author{Andrea Tosin}
\address{Department of Mathematical Sciences ``G. L. Lagrange'', Politecnico di Torino, Corso Duca degli Abruzzi 24, 10129 Turin, Italy}
\email{andrea.tosin@polito.it}



\begin{abstract}
Aiming to describe traffic flow  on road networks with long-range  driver interactions, we study a nonlinear transport equation  defined on an oriented network where the velocity field depends not only on the state variable but also on the distribution of the population. We prove existence, uniqueness and continuous dependence results of the solution intended in a suitable measure-theoretic sense. We also provide a representation formula in terms of the push-forward of the initial and boundary data along  the network and discuss an explicit example of nonlocal velocity field fitting our framework.
\end{abstract}

\maketitle
\section{Introduction}
In recent times there has been a considerable amount of literature devoted to the study of  evolution equations  in measures spaces. Indeed the measure-valued approach presents,  with respect to other approaches based on classical and weak solutions, some significant advantages: the population is represented by a probability distribution, providing a unified framework for both discrete and continuous models; short and long range interaction mechanisms are efficiently described by taking a velocity field depending on local terms, determined by  the geometry of the space, and nonlocal terms, regulated by the position of the other individuals, hence on the whole measure; aggregation phenomena that in a classical setting lead  to blow-up of the solution are plainly taken into account by the measure setting. The by now classical reference for evolution equation in measure spaces is the book \cite{ambrosio2008BOOK}, while we refer to \cite{canizo,cristiani2014BOOK,Frasca_Tosin,Piccoli_Rossi,piccoli2011ARMA} for various applications to the study of complex phenomena. However most of the literature about measure-valued equations considers these problems in the full space, because their study in bounded domains poses additional difficulties due in particular to the interpretation of the boundary conditions. For the specific case of a bounded interval, an interpretation of the boundary condition in a measure sense has been pursued in \cite{evers2015JDE,evers2016SIMA}, while in \cite{Gwiazda} a measure-valued transport equation on a sequence of intervals with a transmission condition at intersection points is considered.

Motivated by pedestrian and vehicular traffic modelling in urban areas, several models have been proposed for traffic flow on road networks, see \cite{Bressan,garavello2016BOOK,Garavello_Piccoli} and references therein. Most of these models are based on a fluid-dynamical approach and take into account only local interactions among drivers, the main purpose being to find appropriate rules at the junctions, namely the vertices of the network, to optimize the traffic flow.\par
In order to extend the measure-valued approach to networks, in  \cite{Cam_Dem_Tos} it was studied
the  linear transport equation
\begin{equation}\label{intro1}
	\p_t m + \p_x (v(x) m) = 0 \qquad \text{in } \Gamma\times [0,T]
\end{equation}
where $\G$ is an oriented network. Existence, uniqueness and continuous dependence results for the measure-valued solution to \eqref{intro1} were provided, along with a local representation formula on each arc. Even if this simplified model already presents some interesting peculiarities and difficulties due to the presence of the junction conditions,
nonlocal driver interactions were not included in the model since $v$ was assumed to depend only on the space variable.

The aim of this paper is to study measure-valued nonlinear transport equations on networks where the velocity field depends on the measure itself.  More precisely, we consider the nonlinear transport equation
\begin{equation}\label{intro2}
	\p_t m+ \p_x(v[m_t]m) = 0, \qquad \text{in } \Gamma\times [0,T]
\end{equation}
where the velocity  $v$ still depends on the $x$-variable, but also on the vehicle distribution $m_t$ at time $t$. To explain the main difference between \eqref{intro1} and \eqref{intro2} we observe that \eqref{intro1} is formally equivalent to a system of equations, one for each arc, coupled via the transmission conditions at the vertices. Instead, in \eqref{intro2} the evolution equation in each arc does not only depend on the distribution of the vehicles flowing into the arc from the junction but also on the (global) distribution $m_t $ at time $t$ on $\G$.

To show the well posedness of \eqref{intro2} we approximate the nonlinear transport equation by a sequence of linear problems obtained via a semi-discrete-in-time approximation of \eqref{intro2}. We define a partition of the time interval $[0,T]$ in a family of subinterval $[t_k, t_k+\Delta t]$ and on each of these intervals we solve the linear problem \eqref{intro1} with the nonlinear velocity $v[m_t]$ replaced by the linear one $v[m_{t_k}]$. In such a way we obtain a sequence of measure $\{m^{\Delta t}\}$ defined on $[0,T]$. Using the results on the linear problem, we prove that for $\Delta t \to 0^+$ the sequence $\{m^{\Delta t}\}$ converges (upon subsequences) to a measure $m\in\cM^+(\G\times [0,T])$ which is a solution of \eqref{intro2}. A continuous dependence result and a representation formula in terms of the push-forward of the initial and boundary data along the admissible paths on the network complete the study of \eqref{intro2}. We also analyze a specific example of velocity field to show that the measure approach allows us to consider   some significant aspects in the model such as local and nonlocal interactions, source data, statistical knowledge of the driver behaviour at junctions.

In more details, the paper is organized in the following way. In Section \ref{sec:notation} we recall some basic notations and preliminary definitions, while in Section \ref{sec:linear} we review  the  results proved in \cite{Cam_Dem_Tos}. In Section \ref{sec:existence} we introduce  the semi-discrete  approximation scheme and prove its convergence. Finally, in Section \ref{trafficmodel} we analyze a specific velocity field suitable for  vehicular traffic over a road network and  satisfying the setting of the paper.


\section{Notations and preliminary definitions} \label{sec:notation}
This section is devoted to notations and definitions that we shall use in the sequel. Some of these definitions are classical but not necessarily standard, thus we recall them for reader's convenience.

In our model, the distribution of particles on the network is represented  by a positive measure, hence we   introduce  an appropriate topology for the space of  measures. Let   $\cT$ be  a topological space endowed  with a distance  $d:\cT\times\cT \to \R$.
Define the norm $\|\cdot\|_{BL}$ as
$$\|\phi\|_{BL} = \|\phi\|_{\infty} + |\phi|_{L},$$
where
$$|\phi|_L = \sup_{\substack{x,\,y\in\cT \\ x\ne y}} \frac{|\phi(x) - \phi(y)|}{d(x,y)}$$
and let $BL(\cT)$ be  the Banach space of bounded and Lipschitz continuous functions equipped with the norm $\|\cdot\|_{BL}$. Let moreover $\cM(\cT)$ denote the space of finite Radon measures on $\cT$. We define a norm on this space by
\begin{equation}\label{norm}
\|\mu\|_{BL}^\ast = \sup_{\substack{\phi \in BL(\cT) \\ \|\phi\|_{BL}\leq 1}}|\dual{\mu}{\varphi}|.
\end{equation}
It is easy to see that if $\mu\in\cM^+(\cT)$ then $\normBL{\mu}=\mu(\cT)$. Moreover, even if the space $(\cM(\cT),\,\normBL{\cdot})$ is in general \emph{not} complete, the cone $\cM^+(\cT)$ is complete since it is a closed subset.

We observe that other norms  on $\cM(\cT)$, such as the total variation norm, may not be well suited for  transport problems where one wants to measure the distance between flowing mass distributions. Indeed with these norms the distance between two distinct Dirac masses is independent of the distance of their supports.

In the following we will employ  the property that a measure  $\mu\in\cM^+(\cT)$ can be represented as a (continuous) sum of elementary masses in the form
$$ \mu=\int_{\cT}\delta_x\,d\mu(x),$$
where $\int_\cT$ is intended as a Bochner integral (see \cite{evers2015JDE, evers2016SIMA}). Actually the previous formula suggests that to obtain some properties of a measure-valued solution to an evolution problem it is sufficient to study the corresponding propagation of a Dirac measure.

A \emph{network} $\Gamma=(\cV,\,\cE)$ is given by a finite collection of vertices $\cV:=\{x_i\}_{i\in I}$ and a finite collection of continuous non-self-intersecting arcs $\cE=\{e_j\}_{j\in J}$ whose endpoints belong to $\mathcal{V}$. Several parametrizations of the arcs in $\Gamma$ can be introduced; for our purposes every bounded arc $e_j \in \mathcal{E}$ is parametrized by a smooth injective function $\pi_j:[0,L_j]\subset\R\to\R^d$, where $L_j\in \R^+$. Alternatively, if $e_j$ is an unbounded arc terminating in a vertex $V$ we parametrize it by a smooth injective function $\pi_j : (-\infty, 0] \to \mathbb{R}^d$ such that $\pi_j(0) = V$; if instead it is an arc originating from $V$ we define the parametrization on $[0, +\infty)$ in such a way that $\pi_j(0) = V$. We assume that $\G$ is connected and oriented and that the maps $\{\pi_j\}_{j \in J}$ comply with the orientation of $\G$, i.e. if $x_i,x_j \in \cV$ are the vertices of an arc $e_k \in \cE$ oriented from $x_i$ to $x_j$, then  $\pi_k(0) = x_i$ and $\pi_k(L_k)=x_j$. To each function $\phi$ defined on defined on $\prod_{j\in J}e_j$ we associate the projection $(\phi_j)_{j\in J}$ defined on the parameter space as
\[
\phi_j(s):=\phi(\pi_j(t)) \qquad s\in[0,L_j], \quad j\in J.
\]
The integral of $\phi$ on $\G$ is naturally defined as
\[
\int_\G \phi(x)dx=\sum_{j\in  J}\int_{e_j}\phi(x)\,dx=\sum_{j\in J}\int_0^{L_j}\phi_j(s)ds.
\]
We provide $\Gamma\times [0,\,T]$  with the distance
$$ d_{\G}(x,\,y)+\abs{t-s}, \qquad (x,\,t),\,(y,\,s)\in\Gamma\times [0,\,T]  $$
where $d_{\G}$ is the minimum path distance on the network $\Gamma$, and we consider the spaces $\cM^+(\Gamma\times [0,\,T])$ and $\cM^+(\Gamma)$ endowed with the corresponding norms defined in \eqref{norm}. Throughout the paper we consider measures without Cantorian  part, so that for $\mu \in \cM^+(\G)$ the pairing
$$ \langle\mu,\phi\rangle := \int_\G\phi(x)d\mu(x) $$
is well defined for every $\phi\in BV(\G)$. 
The Cantorian measures are excluded because, for the application we are considering, this kind of measure does not have any significant interpretation.

Given a vertex $x_i\in\cV$, we say that an arc $e_j\in\cE$ is \emph{outgoing} (respectively, \emph{incoming}) if $x_i=\pi_j(0)$ (respectively, if $x_i=\pi_j(L_j)$). We denote by $\text{Out}(x_i)$ and  $d_O^{x_i}$  the set and the number of outgoing arcs and by  $\text{Inc}(x_i)$  and  $d_I^{x_i}$  the set and the number  of incoming arcs   in $x_i$. We set  $d^{x_i}:=d_I^{x_i}+d_O^{x_i}$ and we say that a vertex $x_i$ is \emph{internal} if $d_I^{x_i}\cdot d_O^{x_i}>0$,   a \emph{source} if $d_O^{x_i}=d^{x_i}$ and a \emph{sink} if $d_I^{x_i}=d^{x_i}$. Moreover, without loss of generality, we assume that for each source there exists a unique outgoing arc, i.e. $d_{O}^{x_i} = 1$.

In the model discussed in the paper the sources represent the vertices where the particles enter the network while the sinks are the vertices where they leave the network. Since the velocity may depend on the distribution of the particles on the whole network, in order to simplify the notation we prefer to consider a network without sinks, i.e. such that the terminal arcs always have infinite length. In any case, at the expense of a heavier notation, it is not difficult to include in the model also the contribution of the sinks.

From now on we denote by $\cS$ the subset of sources. These vertices represent the boundary of the network and we prescribe a \emph{boundary measure }
$$\sigma_0 = \sum_{x_i \in \cS}\sigma^i_0$$
where $\sigma^i \in \cM^+(\{x_i\}\times[0,T])$. We also  prescribe the \emph{initial mass distribution} of the particles in $\Gamma$ as a positive measure
$$ m_0=\sum_{j \in J}m_0^j $$
with $\supp{m_0^j}\subseteq e_j$.

\begin{remark}
Instead of the $\sigma_0^i$'s, another possibility is to consider as boundary data the flow measures $q_0^i = v^i_0 \cdot \sigma^i_0 \in \cM^+([0,T]\times \{x_i\})$, where the $v^i_0$'s are the velocities in $BV([0,T]; \mathbb{R}^+)$ at the sources. 

\end{remark}

As usual when dealing with differential equations on networks, the transition conditions at the vertices play a crucial role since they model the different behaviour of the particles at the junctions (e.g., traffic lights, priority rules). We consider a $J\times J$  \emph{distribution matrix} $\{p_{kj}(t)\}_{k,\,j=1}^{J}$ such that $P(t)$ is a stochastic matrix for every $t \in [0,T]$, i.e.
\begin{equation} \label{eq:pijk.sum.1}
\begin{split}
	 0\leq p_{kj}(t)\leq 1,\qquad \sum_{j=1}^{J}p_{kj}(t)=1\\
    p_{kj}(t)=0\quad \text{if either $e_k\cap e_j=\emptyset$ or $e_j\to e_k$.}
    \end{split}
\end{equation}
where $e_j\to e_k$ means that $e_j$ comes before $e_k$ in the assigned  orientation of the network. Here $p_{kj}(t)$ represents the fraction of mass which at time $t$ flows from the arc $e_k$ to the arc $e_j$ and \eqref{eq:pijk.sum.1} implies that the mass cannot concentrate at the vertices. Since we consider measures $m\in\cM^+(\G\times [0,T])$ without Cantorian part, we assume that $p_{kj}\in BV([0,T])$ so that $p_{kj}\cdot m$ has no Cantorian part as well.

In order to describe the transport of the measures on the network, we introduce a \emph{nonlinear velocity field} $v: \cM^+(\G) \times \G \to \R$ with the following properties:
\begin{itemize}
\item[(H1)] $v$ is nonnegative and has a maximum value $V_{max}>0$;
\item[(H2)] $v$ is Lipschitz continuous with respect to the state variable, i.e. on each arc $e_j \in \cE$
$$|v[m](x) - v[m](y)|\leq L|x - y|,\quad \text{$\forall m \in \cM^+(\G)$, $x,y \in e_j$.}$$
\item[(H3)] $v$ is Lipschitz continuous with respect to the measure, i.e.
$$|v[m_1](x) - v[m_2](x)| \leq L \|m_1 - m_2\|_{BL}^\ast \quad\text{$\forall x \in \G$, $m_1,m_2 \in \cM^+(\G)$.}$$
\end{itemize}
It is important to observe that we do not require the continuity of the velocity field on the whole network but only inside the edges. Note also that the dependence of $v$ on the measure $m$ is global, i.e. the velocity depends on the entire support of $m$ on $\G$.

When considered on a single arc isomorphic to $\R$, the previous assumptions coincide with the ones for the corresponding nonlinear transport model in \cite{cristiani2014BOOK}. Moreover, for a fixed $m\in\cM^+(\G)$, the velocity field $v[m]$ satisfies the hypotheses of the linear transport problem considered in \cite{Cam_Dem_Tos}.

We conclude this section with a notion of $p$-moment for finite measures on networks. This is a straightforward generalization of the corresponding concept in the Euclidean space and we give some details for the reader's convenience.
\begin{definition}
Let $p\in\N$ and $x\in\G$. The $p$-moment centered at $x$ of a finite measure $m\in\cM^+(\G)$ is defined by
\begin{equation}\label{momentum}
\dual m{d_{\G}(\cdot,x)^p} := \int_{\G}d_{\G}(y,x)^p dm(y).
\end{equation}
\end{definition}

\begin{lemma}\label{lemmoment}
A finite measure $m\in\cM^+(\G)$ has finite $p$-moment if and only if it has finite $p$-moment on every arc $e_j\in\cE$ such that $\mathcal{L}(e_j)=+\infty$.
\end{lemma}
\begin{proof}
Assume without loss of generality that $x=x_i \in  \cV$ and set  $d(\cdot) = d_{\G}(\cdot, x_i)$. Given a measure $m \in \cM^+(\G)$,  $m = \sum_{j\in J} m^{j}$ with   $\supp\{m^j\}\subseteq e_j$, we can write
\begin{align*}
\langle m; d^p\rangle = \sum_{\substack{j\in J \\ \mathcal{L}(e_j) < +\infty}}\langle m^{j}; d^p\rangle + \sum_{\substack{j\in J \\ \mathcal{L}(e_j) = +\infty}}\langle m^{j}; d^p\rangle.
\end{align*}
If $e_j \in \mathcal{E}$ has finite length, then $d(\cdot)$ has its maximum value $\overline d_j$ on $e_j$. Then for
$$ \overline d := \max\limits_{\substack{j\in J \\ \mathcal{L}(e_j) < +\infty}}\overline d_j $$
we have
$$\sum_{\substack{j\in J \\ \mathcal{L}(e_j) < +\infty}}\langle m^{j}; d^p\rangle \leq \sum_{\substack{j\in J \\ \mathcal{L}(e) < +\infty}}
\overline d_j\cdot m^j(e_j)\le   \overline d\cdot m(\G).$$
On the other hand, if $\mathcal{L}(e_j)=+\infty$ and $e_j=\pi_j([0,+\infty))$ with $x_k=\pi_j(0)\in \cV$, by Jensen's inequality we have
\begin{align*}
  \langle m^{j}; d^p \rangle&= \int_{[0,+\infty)}(|y| + d(x_k))^p dm^{j}(y)\\
&\leq 2^{p-1}\int_{[0,+\infty)}|y|^p dm^{j}(y) +2^{p-1}d(x_k)^p m(e_j).
\end{align*}
By the last inequality, the statement easily follows.
\end{proof}

The finite $p$-moment property of a measure $m$ is clearly independent of the point $x\in\G$ fixed in the definition \eqref{momentum}.

\section{The linear transport problem  }\label{sec:linear}
The aim of this section is twofold. In the first part, we briefly review the results for the linear problem in \cite{Cam_Dem_Tos}, since they are an important tool for developing the theory of the nonlinear problem via an approximation procedure. Hence, we give a new representation formula for the measure-valued solution of the linear problem (afterwards extended also to the nonlinear problem), which generalizes the well-known push-forward formula to the network setting.

In this section we assume that the velocity field $v$ is independent of $m$, i.e. $v[m](x) = v(x)$, and we consider the linear transport problem
\begin{equation}\label{pbl:tl}
	\begin{cases}
		\p_t m + \partial_x(v(x)m) = 0 & \text{in } \G\times[0,T] \\
		m_{t=0} = m_0 \\
		m_{x=x_i} = \sigma^i_0, &\forall x_i\in\mathcal{S} \\
		m^{j}_{x=x_i}=\sum_{k\in \text{Inc}(x_i)}p_{kj}\cdot m^{k}_{x=x_i} & \forall\,x_i\in\mathcal{V}\setminus\mathcal{S},\ \forall\,e_j\in\text{Out}(x_i),
	\end{cases}
\end{equation}
with  $v$, $m_0$, $\sigma_0$  satisfying the assumptions set in Section \ref{sec:notation}.

To explain the meaning of an initial/boundary condition for a measure solution, we recall that, owing to the disintegration theorem (cf.~\cite[Section 5.3]{ambrosio2008BOOK}), we can decompose a measure $m\in\cM^+([0,\infty)\times[0,\infty))$ by projections on the coordinate axes:
\begin{itemize}
\item Using the projection with respect to $x$ we can write
\begin{equation}
	m(dx\,dt)=dm_t(x)\otimes dt,
	\label{eq:dec.space-time}
\end{equation}
where $dt$ is the Lebesgue measure in time in $\R^+_0$ and $dm_t\in\cM^+(\R^+_0\times\{t\})\equiv\cM^+(\R^+_0)$ for a.e. $t\in\R^+_0$.
Hence assigning an initial condition at $t=0$ amounts to prescribing the trace of $m$ on the fiber $\R^+_0\times\{0\}$ according to the decomposition~\eqref{eq:dec.space-time}.\par
\item Similarly, projecting with respect to $t$, we can write
\begin{equation}
	m(dx\,dt)=\frac{dm_x(t)}{v(x)}\otimes dx,
	\label{eq:dec.time-space}
\end{equation}
where $dx$ is the Lebesgue measure in space in $\R^+_0$ and $dm_x\in\cM^+(\{x\}\times\R^+_0)\equiv\cM^+(\R^+_0)$ for a.e. $x\in\R^+_0$. Hence assigning a boundary condition at $x=0$ amounts to prescribing the trace of $m$ on the fiber $\{0\}\times\R^+_0$ according to the decomposition~\eqref{eq:dec.time-space}.
\end{itemize}

In order to give a suitable notion of measure-valued solution to~\eqref{pbl:tl}, we preliminarily introduce integration by parts formulas. Let $C^\infty_{0}(\G\times[0,T])$ be the space of continuous functions in $\G\times[0,T]$ which are infinitely differentiable in $e_j\times[0,T]$ for each $j\in J$ and vanish for $x=\pi_j(t)\to+\infty$ (in local coordinates) if $e_j$ is an unbounded  arc. Then, given $m \in \cM^+(\G\times[0,T])$ and  $f \in C^\infty_0(\G \times [0,T])$,  we define
\begin{equation}\label{int1}
\begin{split}
\dual{\p_t m}f :&= -\dual m{\p_t f} + \dual{m_{T}}{f}- \dual{m_{0}}f\\
&=- \int_0^T \int_\G \p_t f(x,t)d m_t(x)dt + \int_{\G}f(x,T)d m_T(x) - \int_{\G}f(x,0)d m_0(x);
\end{split}
\end{equation}
and
\begin{equation}\label{int2}
\begin{split}
\dual{\p_x(v(x)m)}{f} :&= -\dual m{v(x)\p_x f} - \dual {\sigma_0}f\\
&=-\int_0^T\int_\G v(x) \p_x f(x,t)dm_t(x)dt - \sum_{x_i \in \cS}\int_{[0,T]}f(x_i,t)d\sigma^i_0(t).
\end{split}
\end{equation}
In particular, if $f \in C^\infty_0(\G\times [0,T])$ and $\supp \{f\} \subset e_k\times[0,T]$, then the last formula reads as
\begin{equation}\label{int3}
\begin{split}
&\dual {\p_x(v(x)m)}f  = -\dual {m^k}{v(x)\p_x f} + \dual {m^k_{x=x_i}}f - \dual{m^k_{x=x_j}}f=\\
&- \int_0^T\int_\G v(x) \p_x f(x,t) d m^k_t(x) + \int_{[0,T]}f(x_i,t) d m^k_{x=x_i}(t) - \int_{[0,T]} f(x_j,t) d m^k_{x=x_j}(t),
\end{split}
\end{equation}
where $x_i=\pi_k(0)$, $x_j=\pi_k(L_j)  \in \cV$.
\begin{definition}\label{def:sol}
A measure-valued solution to \eqref{pbl:tl} is a finite measure $m \in \cM^+(\G\times[0,T])$ such that for every $f \in C^1_0(\G\times[0,T]),$
\[
\dual{m_{t=T} - m_0}{f} - \dual{\sigma_0}{f} = \dual{m}{\p_t f + v(x)\p_x f},
\]
and $\forall$ $x_i \in \mathcal{V}\setminus\mathcal{S}$, $\forall$ $e_j \in \text{Out}(x_i)$,
\[
\dual{m^{j}_{x=x_i}}{f} = \sum_{k \in \text{Inc}(x_i)}\dual{m^{k}_{x=x_i}}{p_{kj}f}.
\]
\end{definition}
\begin{remark}
In Definition~\ref{def:sol} a Neumann-type boundary condition on the sinks of $\cW$ is implicitly assumed. Another possibility is that of sticking boundaries considered in \cite{evers2015JDE, evers2016SIMA}.
\end{remark}

In the next theorem we summarize the main results proved in \cite{Cam_Dem_Tos}.
\begin{theorem}\label{existencenet}
There exists a unique measure $m\in\cM^+(\Gamma\times [0,\,T])$ which is a solution to \eqref{pbl:tl} in the sense of Definition \ref{def:sol}. Moreover:
\begin{enumerate}
\item[i)] Given initial data $m_0^1, m_0^2\in\cM^+(\Gamma\times\{0\})$, boundary data $\sigma_0^1,\sigma_0^2\in\cM^+(\cS\times [0,\,T])$ and denoted by $m_1, m_2 \in \cM^+(\Gamma\times[0,T])$ the corresponding solutions, there exists a constant $C=C(T)>0$ such that
\begin{equation*}
	\sup_{[0,T]}	\|m_t^2-m_t^1 \|^{\ast}_{BL}\leq C\left(\normBL{m_0^2-m_0^1}+\normBL{\sigma_0^2-\sigma_0^1}\right).
\end{equation*}
\item[ii)] There exists a constant $C=C(T)>0$ such that
\begin{equation*}
	\normBL{m_t-m_{t'}}+\normBL{\nu_L\llcorner[0,\,t]-\nu_L\llcorner[0,\,t']}\leq C\abs{t-t'}+\sigma_0((t',\,t])
\end{equation*}
for all $t',\,t\in [0,\,T]$ with $t'<t$.
\end{enumerate}
\end{theorem}

The next result is  a   representation formula    which characterizes the solution $m$ of \eqref{pbl:tl} in terms of
the distribution matrix $P(t)$ and of the  push-forward of the initial and boundary data on the paths of the network.
\begin{definition}\label{path}
Given $x\in \G$, a \emph{path} $\gamma$ starting from $x$ is a sequence of edges $(e_{j_0},e_{j_1},\dots,e_{j_n},\dots)$ where $ e_{j_i}\cap  e_{j_{i+1}}=\{x_{j_i}\}\in\cV$ and $e_{j_i}\to e_{j_{i+1}}$ for $i=0,1,2,\dots$; $e_{j_0}$ is the sub-edge  with endpoints $x$ and  $x_{j_0}\in \cV$; the length $\mathcal{L}(\gamma)$ of $\gamma$ is infinite. We denote by $\mathcal{A}(x)$ the set of  paths $\gamma$ starting from $x$.
\end{definition}

Since the network $\G$ is oriented and $\cE$ finite, a path $\gamma$ is necessarily of one of the following two types:
\begin{itemize}
\item $\gamma$ is composed by a finite number of arcs and  the last  one $e_{j_n}$ has infinite length;
\item $\gamma$ is composed by an infinite number of arcs and there exists $n_0, k_0\in \N$ such that for $n\ge n_0$, $\gamma$ is given by a cycle $(e_{j_{n_0}},\dots,e_{j_{n_0+k_0}})$ with $e_{j_{n_0+k_0}}=e_{j_{n_0}}$.
\end{itemize}
We denote by $\Phi^\gamma$ the flow map associated to the velocity field $v$ restricted to $\gamma$, i.e.  $\Phi_s^\gamma(x,s) = x$ and  there are
$t_0:=s<t_1<\dots<t_n<\dots$ such that for every $m=0,1,\dots$ we have $\Phi^\gamma([t_m,t_{m+1}])\subset e_{j_m}$ and
\begin{equation}\label{hidden_ODE}
\frac{d}{dt}\Phi^{\gamma}_t(x,s) = v (\Phi^{\gamma}_{t}(x,s)), \qquad t \in[t_m,t_{m+1}).
\end{equation}
We define the exit times from the arc $e_{j_k}=\pi_{j_k}([0, L_{j_k}])$ of $\gamma$  as
\begin{align*}
&\theta^\gamma_{0}(x,s)=\inf\{t\ge s:\, \Phi^{\gamma}(x,s)=\pi_{j_0} (L_{j_0})\}, \\
&\theta^\gamma_{k}(x,s)=\inf\{t\ge \theta^\gamma_{k-1}(x,s):\, \Phi^{\gamma}(x,s)=\pi_{j_k} (L_{j_k})\},\quad k\geq 1
\end{align*}
and we associate to each $(x,s)\in \G\times[0,T]$ and to each $\gamma \in \mathcal{A}(x)$ a coefficient $p_{\gamma}(x,s) \in [0,1]$ defined by
\begin{equation}\label{pgamma}
p_{\gamma}(x,s) := \prod_{k } p_{{j_k}{j_{k+1}}}(\theta^\gamma_{k}(x,s)),
\end{equation}
where $ p_{{j_k}{j_{k+1}}}$ are the entries of the distribution matrix $P$ defined in \eqref{eq:pijk.sum.1}. The coefficient   $p_{\gamma}(x,s)$ can be interpreted as the fraction of the total mass transported along the path $\gamma$. Due to the properties of $P$, it follows that
\begin{equation*}
p_{\gamma}(x,s) \in [0,1], \qquad \sum_{\gamma \in \mathcal{A}(x)}p_{\gamma}(x,s)=1.
\end{equation*}
\begin{theorem}\label{existence}
If $m  \in \cM^+(\G\times[0,T])$ is a solution of \eqref{pbl:tl}, then for any $t\in [0,T]$, $m_t$ is given by
\begin{equation}\label{globchar}
m_t = \int_\G \sum_{\gamma \in \mathcal{A}(x)}\delta_{(\Phi^\gamma_t(x,0),t)}p_{\gamma}(x,0)dm_0(x)
	+\sum_{x_i \in \cS}\int_{[0,t]}\sum_{\gamma \in \mathcal{A}(x_i)}\delta_{(\Phi^\gamma_{t}(0,s),t)}p_{\gamma}(0,s)d\sigma_0^i(s).
\end{equation}
\end{theorem}
In order to prove the representation formula~\eqref{globchar} we preliminarily recall a characterization of the traces of the solution $m$ of \eqref{pbl:tl} on the fibers $e_j\times\{t\}$ and $\{x_i\}\times[0,\,t]$, where $x_i=\pi_j(L_j)$, in terms of the transport of the initial and boundary data inside $e_j$ (see \cite{Cam_Dem_Tos}).

\begin{lemma}
Let $e_j \in \cE$, then
\begin{align}
	m^j_{t} =\int_{[0,\,\max\{0,\,\tau^{-1}(t)\}]}\,\delta_{\Phi_t(x,\,0)}\,dm_0^j(x)
		+\int_{(\max\{0,\,\varsigma^{-1}(t)\},\,t]}\delta_{\Phi_t(0,\,s)}\,dm_{x=\pi_j(0)}^j(s)\label{candidate}\\
	m^j_{x=\pi_j(L_j)} =\int_{(\max\{0,\,\tau^{-1}(t)\},\,L_j]}\delta_{\tau(x)}\,dm_0^j(x)
		+\int_{[0,\,\max\{0,\,\varsigma^{-1}(t)\}]}\delta_{\varsigma(s)}\,dm_{x=\pi_j(0)}^j(s),
	\label{outnu}
\end{align}
where $\Phi$ is the flow map associated to the velocity $v$ over $e_j\in\cE$ and $\tau$, $\varsigma$ are defined as
\begin{align}
	\tau(x) &= \inf\{s\geq 0\,:\,\Phi_s(x,0)=L_j\}, \label{tau} \\
	\varsigma(t) &= \inf\{s\geq t\,:\,\Phi_s(0,t)=L_j\} \nonumber.
\end{align}
\end{lemma}

\begin{proof}[Proof of Theorem \ref{existence}]
We can observe that, since the velocity $v$ is uniformly bounded on $\G$, we can restrict the proof to the case of networks with a single junction. The general case can be easily obtained adding new arcs and repeating the same argument.

Hence, we consider a simple   network  with  $\cV$ given by   two vertices $\{S,V\}$, where $S$ is a source and $V$ is an internal vertex, and $\cE$ given by an   arc  $e_1$ connecting $S$ to $V$  and by   $n-1$ unbounded arcs $e_k\in \text{Out}(V)$. Due to this choice, we observe that all the paths $\gamma \in \mathcal{A}(x)$ are subsets of $(e_1, e_k)$ if $x \in e_1$, otherwise they are subsets of $e_k$ if $x \in e_k$.

The solution can be written as $m = \int_0^T \delta_{(x,t)}dm_t(x)dt$, where $m_t = \sum_{k=1}^n m^k_t.$

If $k=1$,
by \eqref{candidate} with $S=\pi_1(0)$ and $V=\pi_1(L_1)$, the solution restricted to $e_1$ is given by
\begin{align*}
	m^1_{t} =\int_{[0,\,\max\{0,\,(\tau_1)^{-1}(t)\}]}\,\delta_{\Phi^{e_1}_t(x,\,0)}\,dm_0^1(x)
		+\int_{(\max\{0,\,(\varsigma_1)^{-1}(t)\},\,t]}\delta_{\Phi^{e_1}_t(0,\,s)}\,d \sigma_0(s); \\
\end{align*}
otherwise for $k \in \{2,\hdots, n\}$, by $\eqref{candidate}$ with $V=\pi_k(0)$, on $e_k$ it is given  by
\[
 m^k_t  = \int_{[0, \max\{0,(\tau_k)^{-1}(t)\}]} \,\delta_{\Phi^{e_k}_t(y,\,0)}dm^k_0(x) + \int_{(\max\{0, (\varsigma_k)^{-1}(t)\},t]}\delta_{\Phi^{e_k}_t(0,\,s)}d m_{x=V}^k(s).
\]
Observe that the first term at the right hand-side of the previous equation is the push-forward of the mass $m_0^k$ which at time $t=0$ is in $e_k$; the second term is a percentage of the mass in $e_1$ which flows in $e_k$.

Using the transmission condition $m_{x=V}^k = p_{1k}\cdot m_{x=V}^{1}$ and recalling  that by \eqref{outnu} we have
\[
m^1_{x=V} =\int_{(\max\{0,\,(\tau_1)^{-1}(T)\},\,L_1]}\delta_{\tau_1(x)}\,dm_0^1(x)
 	+\int_{[0,\,\max\{0,\,(\varsigma_1)^{-1}(T)\}]}\delta_{\varsigma_1(s)}\,d\sigma_0(s),
\]
we get
\begin{equation}\label{pz2}
\begin{split}
\int_{(\max\{0, (\varsigma_k)^{-1}(t)\},t]}&f(\Phi_t^{e_k}(0,s),t)d m_{x=V}^k(s) \\
&=\int_{(\max\{0, (\tau_1)^{-1}(t)\}, L_1]}f(\Phi^{e_k}_t(0, \tau_1(x)),t)p_{1k}(\tau_1(x))dm^1_0(x)\\
&+ \int_{[0, \max\{0, (\varsigma_1)^{-1}(t)\}]}f(\Phi^{e_k}_t(0, \varsigma_1(s)), t)p_{1k}(\varsigma_1(s))d\sigma_0(s)\\
&=\int_{(\max\{0, (\tau_1)^{-1}(t)\}, L_1]}f\big(\Phi^{e_k}_t(\Phi^{e_1}_{\tau_1(x)}(x,0), \tau_1(x)),t\big)p_{1k}(\tau_1(x))dm^1_0(x) \\
&+ \int_{[0, \max\{0, (\varsigma_1)^{-1}(t)\}]}f(\Phi^{e_k}_t(\Phi_{\varsigma_1(s)}(0,s), \varsigma_1(s)), t)p_{1k}(\varsigma_1(s))d\sigma_0(s)\\
&= \int_{(\max\{0, (\tau_1)^{-1}(t)\}, L_1]}f(\Phi^\gamma_t(x,0),t)p_{1k}(\tau_1(x))dm^1_0(x) \\
&+ \int_{[0, \max\{0, (\varsigma_1)^{-1}(t)\}]}f(\Phi^\gamma_t(0,s), t)p_{1k}(\varsigma_1(s))d\sigma_0(s),
\end{split}
\end{equation}
Lastly, we observe that $m_t^1$ can be split in $n-1$ parts according to the distribution terms $p_{1k}$. Indeed, if we write $m_t^1 = \sum_{k=2}^{n}(p_{1k}\circ\theta^\gamma_0)\cdot m^1_t$, then
$$ m_t = m^1_t + \sum_{k=2}^n m_t^k = \sum_{k=2}^n (((p_{1k}\circ\theta^\gamma_0)\cdot m^1_t) + m^k_t). $$
Concerning the first term, observing that   $\tau_1(x) = \theta^\gamma_0(x,0)$ and $\varsigma_1(s) = \theta_0(0,s)$, we compute for any $f \in C^\infty_0(\G\times[0,T])$
\begin{equation}\label{pz1}
\begin{split}
\dual{ (p_{1k}\circ\theta^\gamma_0)m^1_t}{f} &= \int_{[0, \max\{0,(\tau_1)^{-1}(t)\}]} f(\Phi^{e_1}_t(x,0),t)p_{1k}(\theta^\gamma_0(x,0))dm^1_0(x) \\&+ \int_{(\max\{0, (\varsigma_1)^{-1}(t)\},t]}f(\Phi_t^{e_1}(0,s),t)p_{1k}(\theta^\gamma_0(0,s))d\sigma_0(s)\\&=\int_{[0, \max\{0,(\tau_1)^{-1}(t)\}]} f(\Phi^{e_1}_t(x,0),t)p_{1k}(\tau_1(x))dm^1_0(x) \\&+ \int_{(\max\{0, (\varsigma_1)^{-1}(t)\},t]}f(\Phi_t^{e_1}(0,s),t)p_{1k}(\varsigma_1(s))d\sigma_0(x).
\end{split}
\end{equation}

By \eqref{pz2}, \eqref{pz1} it follows that, from the parametrization used for each arc,
\begin{align*}
\dual{(p_{1k}\circ\theta^\gamma_0)\cdot m^1_t}{f} + \dual{m^k_t}{f} &= \int_{e_1} f(\Phi_t^\gamma(x,0), t) p_{1k}(\theta^\gamma_0(x,0))dm^1_0(x) +\int_{e_2} f(\Phi^\gamma_t(x,0),t)dm^k_0(x)\\
&+\int_{[0,t]}f(\Phi^\gamma(0,s),t)p_{1k}(\theta^\gamma_0(0,s))d\sigma_0(s).
\end{align*}
If we sum the previous formula over $e_k\in \text{Out}(V)$ we have
\begin{equation}\label{1ton}
\dual{m_t}{f} = \int_{[0,t]}\sum_{\gamma \in \mathcal{A}(x)}f(\Phi^\gamma_t(0,s),t)p_{\gamma}(0,s)d\sigma_0(s) + \int_{\G}\sum_{\gamma \in \mathcal{A}(x)}f(\Phi^\gamma_t(x,0),t)p_{\gamma}(x,0)dm_0(x),
\end{equation}
Hence we have proved formula \eqref{1ton} in the special case of the simple network as above. If we consider the same network with multiple sources, we sum \eqref{1ton} over $x_i \in \cS$ and have the thesis.

Finally, the case of more complex networks can be addressed by replacing in the last part the arc $i$ in $\gamma=(x,e,V,i)$ with another part and then repeating the argument used in this proof.
\end{proof}

\begin{remark}
We observe that formula~\eqref{globchar} is equivalent to the superposition principle introduced in \cite{ambrosio2008BOOK}. Indeed, assuming for simplicity that $\mathcal{S}=\emptyset$, there is a bijective correspondence among paths, as in Definition \ref{path}, and trajectories in \eqref{hidden_ODE}: given $\gamma \in \mathcal{A}(x)$, then $ \gamma = \Phi^{\gamma}(x,[0,T])$. Therefore
$$\mathcal{A}(x) = \{\Phi \in AC([0,T], \Gamma): \Phi \hspace{1 mm}\text{satisfies}\hspace{1 mm} \eqref{hidden_ODE}\}.$$
Let us define $\eta_x \in \cM^+(\mathcal{A}(x))$ by
$$ \eta_{x} = \sum_{\Phi \in \mathcal{A}(x)} p_{\Phi(x,[0,T])}(x,0)\delta_{\Phi(x,0)}. $$
Then \eqref{globchar} can be written as
\begin{align*}
m_t &= \int_\G \int_{\mathcal{A}(x)} \delta_{\Phi_t(x,0)} d\eta_x(\Phi) dm_0(x) \\&= \int_\G \int_{\mathcal{A}(x)} \delta_{(\Phi\circ e_t)(x,0)} d\eta_x(\Phi) dm_0(x)\\&= \int_{\G \times \mathcal{A}(\G)} \delta_{\Phi(x,0)}d(e_t \# \eta),
\end{align*}
where $e_t$ is the evaluation in $t$, i.e. $e_t(\Phi) = \Phi_t$, and $\eta \in \cM^+(\G\times\mathcal{A}(\G))$ is defined by $\eta = m_0 \otimes \eta_x$. This is exactly the form of the superposition principle mentioned above.
\end{remark}

\section{The nonlinear transport problem}\label{sec:existence}
This section is devoted to the study of the nonlinear transport problem, i.e. the transport problem  with the velocity field depending on the distribution of the particles on the network.

We consider the problem
\begin{equation}\label{pbl:ntl}
	\begin{cases}
		\p_t m+\partial_x(v[m_t]m)=0 & \text{in } \G\times[0,T] \\
		m_{t=0}=m_0 \\
		m_{x=x_i}=\sigma^i_0 & \forall\,x_i\in\mathcal{S} \\
		m^{j}_{x=x_i}=\sum_{k\in\text{Inc}(x_i)}p_{kj}\cdot m^{k}_{x=x_i} & \forall\,x_i\in\mathcal{V}\setminus\mathcal{S},\ \forall\,e_j\in\text{Out}(x_i)
	\end{cases}
\end{equation}
with $v[m]$, $m_0$, $\sigma_0$ satisfying the assumptions stated in Section \ref{sec:notation}. Concerning the interpretation of the initial/boundary conditions and the definition of solution, we argue as in Section \ref{sec:linear}. In particular, thanks to the integration by part formulas \eqref{int1}-\eqref{int2} (with $v(x)$ replaced by $v[m](x)$) we can state:
\begin{definition}\label{def:sol_nonlinear}
A measure-valued solution to \eqref{pbl:ntl} is a finite measure $m \in \cM^+(\G\times[0,T])$ such that for every $f \in C^1_0(\G\times[0,T]),$
$$ \dual{m_{t=T} - m_0}{f} - \dual{\sigma_0}{f} = \dual{m}{\p_t f + v[m_t]\p_x f}, $$
and $\forall\,x_i\in\mathcal{V}\setminus\mathcal{S}$, $\forall\,e_j\in\text{Out}(x_i)$,
\begin{equation}\label{eqn:cond}
\dual{m^{j}_{x=x_i}}{f} = \sum_{e_k \in \text{Inc}(x_i)}\dual{m^{k}_{x=x_i}}{p_{kj}f}.
\end{equation}
\end{definition}

To prove the core result of the paper, i.e. the existence of a measure-valued solution to \eqref{pbl:ntl}, we introduce a semi-discretization-in-time procedure which allows us to approximate the nonlinear problem by a family of linear problems: we define a partition of the time interval $[0,T]$ in a family of subintervals $[t_k, t_k+\Delta t]$ and on each of these intervals we solve the linear problem \eqref{pbl:tl} with the nonlinear velocity $v[m_t]$ replaced by the linear one $v[m_{t_k}]$. In such a way we obtain a sequence of measures $\{m^{\Delta t}\}$ defined on $[0,T]$. Using the results of Section \ref{sec:linear}, we prove that for $\Delta t\to 0^+$ this sequence converges (upon subsequences) to a measure $m\in\cM^+(\G\times [0,T])$, which is a solution of \eqref{pbl:ntl} in the sense of Definition \ref{def:sol_nonlinear}.

Let $N\in\N$, set $\Delta t^{N}:=T/ 2^N$ and define a partition of $[0,T]$ by the intervals $I^N_n := [t^N_n; t^N_{n+1})$, where $t_{n}^N  := n\Delta t^{N}, n= 0, \dots, 2^N$ (in the following we will write $t_n$ in place of $t^N_n$ when the dependence on $N$ is clear by the context). We consider the $2^N$ problems iteratively defined by
\begin{equation*}
	\begin{cases}
		\p_tm+\partial_x(v[m_{t_n}]m)=0 & \text{in } \G\times I^N_n \\
		m_{t=t_n}=m_{t_n}  \\
		m_{x_i\in\mathcal{S}}=\sigma_{0}\llcorner{I^N_n} \\
		m^{j}_{x=x_i}=\sum_{k\in\text{Inc}(x_i)}p_{k j}\cdot m^{k}_{x=x_i} & \forall\,e_j\in\text{Out}(x_i),\ \forall\,x_i\in\mathcal{V}\setminus\mathcal{S},
	\end{cases}
\end{equation*}
where $\sigma_{0}\llcorner{I^N_n}$ is the restriction of $\sigma_0$ to the interval $I^N_n$. We point out that the velocity $v[m_{t_n}]$ is linear on $\G\times I^N_n$ for every $n=0,\dots,2^N-1$. Therefore, owing to Theorem \ref{existencenet}, there exists a unique measure $m^{N,n}\in\cM^{+}(\G\times I^N_n)$ satisfying
\begin{equation}\label{id1}
\dual{m^{N,n}_{t=t^N_{n+1}} - m^{N,n}_{t=t^N_n}}{f} = \dual{ \sigma_{0}\llcorner{I^N_n}}{f} + \int_{t^N_n}^{t^N_{n+1}}\dual{m^{N,n}_{t}}{(\p_t + v[m^{N,n}_{t^N_n}]\p_x)\cdot f(\cdot,t)}dt
\end{equation}
and
\begin{equation}\label{id2}
\dual { (m^{N,n})^j_{x=x_i}\llcorner{I^N_n}} {f} = \sum_{k \in \text{Inc}(x_i)}\dual {p_{kj} (m^{n,N})^k_{x=x_i}{\llcorner{I^N_n}}}{f},
\quad \text{$\forall e_j \in \text{Out}(x_i)$, $\forall x_i \in \mathcal{V}\setminus\mathcal{S}$}
\end{equation}
for every $f \in C^{\infty}_0(\G\times \overline I^N_n)$.

We now denote by $m^N: [0,T] \to \cM^+(\G)$  the map defined  by
\begin{equation}\label{sol_app}
m^N_t=m^{N,n}_t\qquad \text{ for $t \in  I^N_n$, $n= 0, \hdots, 2^N$.}
\end{equation}
We first give some regularity properties of $m^N$ (the proofs of the next two results are postponed to the Appendix \ref{appendix}).
\begin{proposition}\label{lemma1}
For any $t\in [0,T]$, the measure $m^N_t$ is bounded in $(\cM^+(\G), \|\cdot\|_{BL}^\ast)$, uniformly in $N$,   i.e.
there exists  a constant $C = C(T)>0$ such that
\begin{equation}\label{eqnlemma1}
\|m^N_t\|^\ast_{BL} \leq C (\|m_0\|_{BL}^\ast + \|\sigma_0\|^\ast_{BL}), \qquad\forall t\in [0,T].
\end{equation}
Moreover, there exists a constant $C>0$ such that
\begin{equation}\label{lip}
\|m^N_t - m^N_s\|_{BL}^\ast \leq \sigma_0([s,t)) + C |t - s|,
\end{equation}
for every $0\leq s < t \leq T$.
\end{proposition}
\begin{proposition}\label{momentact}
Assume that
\begin{equation}\label{hyp:moment}
m_0 \text{ has finite $p$-moment on } \G, \quad p=1,2.
\end{equation}
Then $m^N_t$ has finite $p$-moment, $p=1,2$, on $\G$ for all $t\in [0,T]$. In particular, there exists a positive constant $C=C(T, V_{max},\|m_0\|_{BL}^\ast,\|\sigma_0\|_{BL}^\ast)$ such that
\[
 \dual{m^N_t}{d^p}  \leq C,
\]
for every $t \in [0,T]$ and every $N \in \N$.
\end{proposition}
Inequality \eqref{lip} shows that the map $t\mapsto m^N_t$ is in general not Lipschitz continuous in $t$ if the source measure $\sigma_0$ is not absolutely continuous with respect to the Lebesgue measure. To prove the convergence of $m^N$,  we need to assume
that $\sigma_0 \in \cM^+(\cS\times[0,T))$ is absolutely continuous with respect to the Lebesgue measure $\mathcal{L}(dt)$ on $[0,T)$ for every source $x_i \in \cS$, i.e.
\begin{equation}\label{hyp:continuity}
  \sigma_0^i  \ll \mathcal{L}(dt)\qquad \forall\,x_i\in\cS,
\end{equation}
but see Remark \ref{rem:source} below for the case in which $\sigma_0$ contains a finite number of atoms.

We have the following existence result for \eqref{pbl:ntl}:
\begin{theorem}\label{exist1}
Assume \eqref{hyp:moment} and \eqref{hyp:continuity}, then the sequence $\{m^N\}_{N\in \N}$ defined in \eqref{sol_app} converges (upon subsequences) to a map $m:[0,T] \to \cM^+(\G)$  in $C([0,T];\cM^+(\G))$.
In particular,
\begin{equation}\label{conv1}
\lim\limits_{N\to +\infty}\sup_{t\in [0,T]}\|m^N_t-m_t\|_{BL}^\ast=0.
\end{equation}
Moreover, the measure $m:=\int_{0}^T\int_\G \delta_{(x,t)}dm_t(x)dt$ is a solution to \eqref{pbl:ntl} in the sense of Definition \ref{def:sol_nonlinear}.
\end{theorem}
\begin{proof}
We proceed in several steps.

\emph{Step (i): Convergence.}
To show that    $\{m^N\}_{N \in \N}$ is relatively compact in $C([0,T], \cM^+(\G))$, it is sufficient to check that the sequence satisfies the conditions of the Ascoli-Arzel\`a criterion in the space of measures (see \cite{ambrosio2008BOOK}): equicontinuity, tightness and uniform integrability. Equicontinuity is a consequence of Proposition \ref{lemma1}, taking into account that by  \eqref{lip}, \eqref{hyp:continuity} $\{m^N\}_{N \in \N}$  is uniformly Lipschitz continuous in $t$. The other two properties, tightness and uniform integrability, are equivalent to uniform estimates of the first and second moments of the measure $m^N_{t}$. These estimates are proved in Proposition \ref{momentact}. Hence we conclude that that, upon subsequences, there exists $m\in C([0,T], \cM^+(\G))$ such that
\begin{equation}\label{conv}
	m^N_t \to m_t  \quad  \forall t\in [0,T]
\end{equation}
for $N\to\infty$.

\emph{Step (ii): $m$ is a solution.}
We now show that the measure $m\in \cM^+(\G\times [0,T])$ defined by
\begin{equation}\label{limit_measure}
m(dxdt) := m_t(dx)\otimes dt = \int_{0}^T\int_{\G}\delta_{(x,t)}dm_t(x)dt ,
\end{equation}
where the  $m_t$ is as in \eqref{conv}, satisfies
\begin{equation}\label{balance}
	\dual{m_T - m_0}{f} - \dual{\sigma_0}{f} = \int_0^T\dual{m_t}{\p_t f(\cdot,t) + v[m_t]\p_x f(\cdot, t)}dt
\end{equation}
for every $f \in C^{\infty}_0(\G\times [0,T])$. Denote $v_{n}^{N} = v[m^N_{t_{n}^{N}}]$.  Summing over $n$ the identities \eqref{id1} and \eqref{id2}, we get that
the measure
$m^N(dxdt) = m^N_t(dx) \otimes dt   \in \cM^+(\G\times[0,T])$ satisfies
\begin{equation}\label{f1}
\dual {m^N_T - m_0}{f} - \dual {f}{\sigma_0} = \sum_{n=0}^{2^N - 1}\int_{I^N_n}\dual {m^N_t}{(\p_t + v_{n}^{N}\p_x)\cdot f(\cdot, t)} dt,
\end{equation}
for every $f \in C^\infty_0(\G\times[0,T])$.
Passing to the limit for $N \to +\infty$ in \eqref{f1}, we first   observe that  by \eqref{conv1} we have for the left hand side
$$\dual {m^N_T - m_0}{f} \to \dual {m_T - m_0}{f}\quad \text{for $N\to\infty$}.$$
To show the convergence of the  right hand side of \eqref{f1}, we  claim that
\begin{align}\label{claim1}
	\begin{aligned}[c]
		& \sum_{n=0}^{2^N-1}\int_{I^N_n}\left(\dual {m^N_t}{(\p_t+v_{n}^{N}\p_x)\cdot f(\cdot,t)}
			-\dual{m_t}{(\p_t+v[m_t]\p_x)\cdot f(\cdot,t)}\right)dt \\
		&= \int_{I^N_n}\dual{m^N_t-m_t}{(\p_t+v[m_t]\p_x)\cdot f(\cdot,t)}dt+\sum_{n=0}^{2^N-1}\int_{I^N_n}\dual{m^N_t}{(v_{n}^{N}
			-v[m_t])\p_x\cdot f(\cdot,t)}dt
	\end{aligned}
\end{align}
tends to $0$ for $N\to \infty$. Indeed for $f\in C^{\infty}_0(\G\times[0,T])$ by \eqref{conv1} we have, when $N\to\infty$,
\begin{align}\label{claim1a}
	\begin{aligned}[c]
		& \sum_{n=0}^{2^N-1}\int_{I^N_n}\dual{m_t-m^N_t}{(\p_t+v[m_t]\p_x)\cdot f(\cdot,t)} dt \\
		&\leq T\sup_{t\in [0,T]}|\dual{m_t-m^N_t}{(\p_t+v[m_t]\p_x)\cdot f(\cdot,t)}|\leq CT\sup_{t\in [0,T]}\|m_t^N-m_t\|_{BL}^\ast\to 0.
	\end{aligned}
\end{align}
Moreover, for every fixed $n=0,\dots,2^N-1$, $t\in I^N_n$ and $x\in\G$ we have
\begin{equation}\label{stima_vel}
|v_{n}^{N}(x) - v[m_t](x)|  \leq L \|m^N_{t^N_n} - m_t\|_{BL}^\ast \leq L\|m^N_{t^N_n} - m^N_{t}\|_{BL}^\ast + L\|m^N_t - m_t\|_{BL}^\ast.
\end{equation}
From Proposition \ref{lemma1} it results
\[
\|m^N_{t^N_n} - m^N_{t}\|_{BL}^\ast \leq \sigma_{0}([t^N_n, t))+ C|t - t^N_n|,
\]
therefore
\begin{align*}
	& \int_{t^N_n}^{t^N_n+\Delta t^N}\dual{m^N_t}{(v_{n}^{N}-v[m_t])\p_x f(\cdot,t)}dt \\
	&\leq \int_{t^N_n}^{t^N_n + \Delta t^N} \left[ C_1(\sigma_{0}([t^N_n, t))+ |t - t^N_n| + \|m^N_t - m_t\|_{BL}^\ast)\dual{m^N_t}{\p_x f(\cdot, t)}\right]dt \\
	&\leq C_1\left( \sigma_0(I^N_n) \Delta t^N + \frac{1}{2} (\Delta t^N)^2 + \Delta t^N \sup_{t \in I^N_n}\|m^N_t - m_t\|_{BL}^\ast\right)\sup_{t \in I^N_n}\|m^N_t\|_{BL}^\ast,
\end{align*}
where $C_1 = \max \{L, CL\}$. 
Again by Proposition \ref{lemma1}, we have the estimate
$$\sup_{t \in [0,T]}\|m^N_t\|_{BL}^\ast < D$$
for a positive constant $D$ independent of $N$. Hence
\[
\int_{I^N_n}\dual {m^N_t}{(v_{n}^{N} - v[m_t])\p_x f(\cdot, t)} dt \leq DC_1 \Delta t^N \left( \sigma_0(I^N_n) + \frac{1}{2}\Delta t^N + \sup_{t \in I^N_n}\|m^N_t - m_t\|_{BL}^\ast\right),
\]
and therefore by  \eqref{conv}
\begin{align}\label{claim1b}
	\begin{aligned}[c]
		& \left|\sum_{n=0}^{2^N-1}\int_{I^N_n}\dual{m^N_t}{(\p_t+(v_{n}^{N}-v[m_t])\p_x)\cdot f(\cdot,t)}dt\right| \\
		& \leq {DC_1}\left(\Delta t^N\sigma_0([0,T])+\frac{T}{2}\Delta t^N+T\sup_{t\in [0,T]}\|m^N_t-m_t\|_{BL}^\ast\right)\to 0 \quad \text{for } N\to +\infty.
	\end{aligned}
\end{align}

By \eqref{claim1a} and \eqref{claim1b} we get \eqref{claim1}. Hence the measure $m$ satisfies the transport equation \eqref{balance}.

\emph{Step (iii): Vertex condition.}
To conclude  that $m$ defined in \eqref{limit_measure} solves \eqref{pbl:ntl} we further need to show that there exist boundary measures $\{m_{x=x_i}^{j}\}_{x_i \in \cV}^{e_j \in \cE}$, defined by
\begin{align}\label{a3}
\dual{  m^{j}_{x=x_i}}{f}= \int_0^T\dual {m^{j}_t}{(\p_t + v[m_{t}]\p_x) f} dt - \dual {m^{j}_T - m^j_0}{f},
\end{align}
if $e_j \in\text{Inc}(x_i)$, or
\begin{align}\label{a4}
 - \dual{  m^{j}_{x=x_i}}{f}= \int_0^T\dual {m^j_t}{(\p_t + v[m_{t}]\p_x) f} dt -  \dual {m^{j}_T - m^j_0}{f},
\end{align}
if $e_j\in\text{Out}(x_i)$. Then, we need to prove that the vertex condition \eqref{eqn:cond} is satisfied.

Let $f$ be a $C^\infty_0(\G\times[0,T])$ function such that there exists a unique vertex $x_i \in \mathcal{V}$ which belongs to the support of $f(\cdot, t),$ for every $t \in [0,T]$. We have previously seen that
$$ \dual{m^N_{T} - m^N_{0}}f  = \dual {\sigma_0} f + \sum_{n=0}^{2^N - 1}\int_{t_n}^{t_{n+1}}\dual{m^N_{t}}{(\p_t + v[m^N_{t_n}]\p_x)\cdot f(\cdot,t)}dt; $$
then, taking into account that the support of $f$ does not contain source vertices, we have
$$ \dual {m^N_T - m_0}{f}= \sum_{n=0}^{2^N - 1}\int_{I^N_n}\dual {m^N_t}{(\p_t + v_{n}^{N}\p_x) f(\cdot, t)}dt; $$
By \eqref{int3}, if $e_j\in\text{Inc}(x_i)$ then
\begin{equation}\label{a1}
\dual{m^{N,j}_{x=x_i}}{f}= \sum_{n=0}^{2^N - 1}\int_{I^N_n}\dual {m^{N,j}_t}{(\p_t + v_{n}^{N}\p_x) f}dt - \dual {m^{N,j}_T - m^{j}_0}{f};
\end{equation}
otherwise, if $e_j\in\text{Out}(x_i)$ then
\begin{equation}\label{a2}
 - \dual{m^{N,j}_{x=x_i}}{f}= \sum_{n=0}^{2^N - 1}\int_{I^N_n}\dual {m^{N,j}_t}{(\p_t + v_{n}^{N}\p_x) f}dt - \dual {m^{N,j}_T - m^{j}_0}{f}.
\end{equation}
Passing to the limit for $N \to +\infty$ in either \eqref{a1} or \eqref{a2}, by \eqref{conv} we get that there exist measures $m_{x=x_i}^j\in\cM(\{x_i\}\times [0,T])$ which satisfy \eqref{a3} or \eqref{a4} and such that $\|m^{N,j}_{x=x_i}-m^j_{x=x_i}\|_{BL}^\ast \to 0$ for $N \to +\infty$. Since by construction $m_{x=x_i}^{N,j} = \sum_{k \in\text{Inc}(x_i)}p_{k e}\cdot m_{x=x_i}^{N,k}$, we get that the same transmission condition \eqref{eqn:cond}  is satisfied by the limit  measure $m$.
\end{proof}

We now extend to the nonlinear transport problem \eqref{pbl:ntl} the representation formula for the solution of the linear problem \eqref{pbl:tl} proved in Theorem\eqref{existence}. Given $m \in \cM^+(\G\times [0,T])$, we denote by $\Phi^\gamma$ the flow map associated to the velocity field $v[m_t]$ restricted to $\gamma$, i.e. $\Phi_s^\gamma(x,s)=x$ and there are $t_0:=s<t_1<\dots<t_n<\hdots$ such that for every $m=0,1,\dots$, we have $\Phi^\gamma([t_m,t_{m+1}])\subset e_{j_m}$ and
\[
\frac{d}{dt}\Phi^{\gamma}_t(x,s) = v[m_t] (\Phi^{\gamma}_{t}(x,s)), \qquad t \in[t_m,t_{m+1}).
\]
\begin{proposition}
If $m\in\cM^+(\G\times[0,T])$ is a solution to \eqref{pbl:ntl} then $m_t$ is given by
\begin{equation}\label{globchar_nl}
	m_t=\int_\G \sum_{\gamma \in \mathcal{A}(x)}\delta_{(\Phi^\gamma_t(x,0),t)}p_{\gamma}(x,0)dm_0(x)
		+\sum_{x_i\in\cS}\int_{[0,t]}\sum_{\gamma\in\mathcal{A}(x_i)}\delta_{(\Phi^\gamma_{t}(0,s),t)}p_{\gamma}(0,s)d\sigma_0^i(s)
\end{equation}
for every $t\in [0,T]$, where the coefficients $p_\gamma$ are defined as in \eqref{pgamma}.
\end{proposition}
\begin{proof}
We observe that from \eqref{stima_vel} it follows
$$ \sup_{x\in e_j}|v_{j}^{N}(x)-v[m_t](x)|\to 0 \qquad \text{for } N\to +\infty,\, j\in J. $$
The previous estimate implies the uniform convergence of the respective flow maps on a given path $\gamma$, hence the convergence of \eqref{globchar} to \eqref{globchar_nl}.
\end{proof}

\begin{proposition}[Continuous dependence]\label{prop:cont.dep}
Given initial data $m_0^1, m_0^2\in\cM^+(\Gamma\times\{0\})$ and boundary data $\sigma_0^1,\sigma_0^2\in\cM^+(\cS\times [0,\,T])$ satisfying the assumptions \eqref{hyp:moment} and \eqref{hyp:continuity}, let $m^1$ and $m^2$ be the corresponding solutions. Then
$$ \sup_{t\in [0,T]}\|m^1_t-m^2_t\|_{BL}^\ast\leq K(\|m^1_0-m^2_0\|_{BL}^\ast+\|\sigma^1_0-\sigma^2_0\|_{BL}^\ast), $$
where $K=K(T)>0$ is a constant.
\end{proposition}
\begin{proof}
For a fixed $x\in\G$ we consider a path $\gamma\in \mathcal{A}(x)$ starting from $x$ and the flow maps $\Phi^{1,\gamma},\,\Phi^{2,\gamma}$ associated to $v[m^1_t],\,v[m_t^2]$, respectively.

Let $f \in BL(\G \times [0,T])$ with $\|f\|_{BL}^\ast \leq 1$, then by formula \eqref{globchar_nl} we have
\begin{align} \label{claim2}
	\begin{aligned}[c]
		\dual{m^1_t-m^2_t}{f} &=
			\int_\G\left(\sum_{\gamma\in\mathcal{A}(x)}f(\Phi^{1,\gamma}_t(x,0),t)p_\gamma(x)dm_0^1(x)\right. \\
		&\phantom{=} \left.-\sum_{\gamma\in\mathcal{A}(x)}f(\Phi^{2,\gamma}_t(x,0),t)p_\gamma(x)dm_0^2(x)\right) \\
		&\phantom{=}+ \sum_{x_i\in\cS}\left(\int_{[0,T]}\sum_{\gamma\in\mathcal{A}(x_i)}f(\Phi^{1,\gamma}_t(x_i,s),t)p_{\gamma}(x_i)d\sigma^1_{x_i}(s)\right. \\
		&\phantom{=} \left.-\int_{[0,T]}\sum_{\gamma\in\mathcal{A}(x_i)}f(\Phi^{2,\gamma}_t(x_i,s),t)p_{\gamma}(x_i)d\sigma^2_{x_i}(s)\right).
	\end{aligned}
\end{align}
To estimate the right-hand side in \eqref{claim2} we rewrite the first term as
\begin{align*}
	& \int_\G\left( \sum_{\gamma \in \mathcal{A}(x)}f(\Phi^{1,\gamma}_t(x,0),t)p_\gamma(x) dm_0^1 (x)
		- \int_\G \sum_{\gamma \in \mathcal{A}(x)}f(\Phi^{2,\gamma}_t(x,0),t)p_\gamma(x) dm_0^2 (x)\right) \\
	&\quad +\int_\G \sum_{\gamma \in \mathcal{A}(x)}f(\Phi^{2,\gamma}_t(x,0),t)p_\gamma(x) d(m^{1}_0 - m_0^2)(x) \\
	&\quad +\int_\G \sum_{\gamma \in \mathcal{A}(x)}\left(f(\Phi^{1,\gamma}_t(x,0),t) - f(\Phi^{2,\gamma}_t(x,0),t)\right)p_\gamma(x) dm_0^1 (x).
\end{align*}
Since $\|f\|_{BL}^\ast \leq 1$ and $\sum_{\gamma \in \mathcal{A}(x)}p_{\gamma}(x) = 1$ for every $x \in \G$, we have the  estimate
\begin{equation}\label{rif1}
\int_\G \sum_{\gamma \in \mathcal{A}(x)}f(\Phi^{2,\gamma}_t(x,0),t)p_\gamma(x) d(m^{1}_0 - m_0^2)(x) \leq \|m^1_0 - m^2_0\|_{BL}^\ast.\\
\end{equation}
Moreover ,
\begin{align*}
|f(\Phi^{1,\gamma}_t(x,0),t) - f(\Phi^{2,\gamma}_t(x,0),t)| &\leq d(\Phi^{1,\gamma}_t(x,0), \Phi^{2,\gamma}_t(x,0)) \leq d_{\gamma}((\Phi^{1,\gamma}_t(x,0), \Phi^{2,\gamma}_t(x,0))),
\end{align*}
where $d_{\gamma}$ is the path distance $d$ restricted to  $\gamma.$ It follows
\begin{align*}
d_{\gamma}(\Phi^{1,\gamma}_t(x,0), \Phi^{2,\gamma}_t(x,0))&\leq \int_0^t \left|v[m^1_s](\Phi^{1,\gamma}_s(x,0)) - v[m^2_s](\Phi^{2,\gamma}_s(x,0))\right|ds\\&\leq\int_0^t L \left(\|m^1_s - m^2_s\|_{BL}^\ast + d_{\gamma}(\Phi^{1,\gamma}_s(x,0), \Phi^{2,\gamma}_s(x,0))\right)ds.
\end{align*}
By Gronwall's inequality,
\begin{equation*}
d_{\gamma}(\Phi^{1,\gamma}_t(x,0), \Phi^{2,\gamma}_t(x,0)) \leq L\left(\int_0^t \|m^1_s - m^2_s\|_{BL}^\ast  ds\right) e^{Lt},
\end{equation*}
and consequently
\begin{equation*}
|f(\Phi^{1,\gamma}_t(x,0),t) - f(\Phi^{2,\gamma}_t(x,0),t)| \leq L\left(\int_0^t \|m^1_s - m^2_s\|_{BL}^\ast  ds\right) e^{Lt}.
\end{equation*}
The previous inequality implies
\begin{align*}
	& \int_\G \sum_{\gamma \in \mathcal{A}(x)}\left(f(\Phi^{1,\gamma}_t(x,0),t) - f(\Phi^{2,\gamma}_t(x,0),t)\right)p_\gamma(x) dm_0^1 (x) \\
	& \leq L\left(\int_0^t \|m^1_s - m^2_s\|_{BL}^\ast  ds\right) e^{Lt}\|m^1_0\|_{BL}^\ast.
\end{align*}
Proceeding in a similar way for the second term in \eqref{claim2}, we   obtain the   inequality
\begin{equation}\label{rif3}
\begin{split}
&\sum_{x_i \in \cS}\left(\int_{[0,T]} \sum_{\gamma \in \mathcal{A}(x_i)} f(\Phi^{1,\gamma}_t(x_i,s),t)p_{\gamma}(x_i)d\sigma^1_{x_i}(s) - \int_{[0,T]} \sum_{\gamma \in \mathcal{A}(x_i)} f(\Phi^{2,\gamma}_t(x_i,s),t)p_{\gamma}(x_i)d\sigma^2_{x_i}(s) \right)  \\
&\leq \|\sigma^1_0 - \sigma^2_0\|_{BL}^\ast + L\left(\int_0^t \|m^1_s - m^2_s\|_{BL}^\ast  ds\right) e^{Lt}\|\sigma^1_0\|_{BL}^\ast
\end{split}
\end{equation}
Using \eqref{rif1}--\eqref{rif3} in \eqref{claim2} we get
$$\dual{m^1_t - m^2_t}{f} \leq \|m^1_0 - m^2_0\|_{BL}^\ast + \|\sigma^1_0 - \sigma^2_0\|_{BL}^\ast + C\int_0^t \|m^1_s - m^2_s\|_{BL}^\ast ds,$$
where $C=L e^{L T} (\|m^1_0\|_{BL}^\ast + \|\sigma_0^1\|_{BL}^\ast)$. Taking the supremum with respect to $f$  we get
\[
\|m^1_t - m^2_t\|_{BL}^\ast \leq (\|m^1_0 - m^2_0\|_{BL}^\ast + \|\sigma^1_0 - \sigma^2_0\|_{BL}^\ast) + C\int_0^t \|m^1_s - m^2_s\|_{BL}^\ast  ds.
\]
Finally, applying again Gronwall's inequality we obtain
\[
\|m^1_t - m^2_t\|_{BL}^\ast \leq (\|m^1_0 - m^2_0\|_{BL}^\ast + \|\sigma^1_0 - \sigma^2_0\|_{BL}^\ast)e^{Ct}. \qedhere
\]
\end{proof}

As an immediate consequence of the continuous dependence result stated in Proposition~\ref{prop:cont.dep} we have
\begin{corollary}
The solution of the nonlinear transport problem \eqref{pbl:ntl}  is unique.
\end{corollary}

\begin{remark}\label{rem:source}
For traffic models on road networks, it is reasonable to consider measures without Cantorian part but the assumption \eqref{hyp:continuity} is quite restrictive since it also excludes the presence of atomic terms in the source measure $\sigma_0$. Recall that \eqref{hyp:continuity} gives the uniform continuity with respect to $t$ of the map $m^N_t$, $t\in [0,T]$, which is necessary in order to apply the Ascoli-Arzel\`{a} criterion. We now explain how to bypass this difficulty and to extend the results of this section,  in particular Theorem \ref{exist1}, to the case of a source measure of the type
$$ \sigma_0 = \sum_{x_i \in \cS}(\sigma_{ac,0}^{x_i} + \sigma_{d,0}^{x_i}), $$
where $\sigma_{ac,0}^{x_i}\ll \mathcal{L}(dt)$ and $\sigma_{d,0}^{x_i}$ is an atomic finite measure in $\cM^+(\cS\times[0,T])$ with a finite number of atoms.

Consider first the case of a source measure $\sigma_{0} = \delta_{(x_i,\tau)}$, for $x_i \in \cS$ and $\tau \in (0,T)$. We can apply Theorem \ref{exist1} in $[0,\tau]$ where  $\sigma_0\equiv 0$ is absolutely continuous with respect to $\mathcal{L}(dt)$   to obtain the existence  of  a solution $m$  to \eqref{pbl:ntl} in $[0,\tau]$. Then we consider    \eqref{pbl:ntl} in $[\tau, T]$   with  initial condition   $m_{\tau} + \delta_{(x_i,\tau)}$ and boundary measure $(\sigma_{0})\llcorner{(\tau,T]}\equiv 0$. Again, since $\sigma_0\equiv 0$ is absolutely continuous with respect to $\mathcal{L}(dt)$ in $[\tau,T]$, we obtain a solution of the problem in $[\tau,T]$. Gluing together  the solutions previously obtained in $[0,\tau]$ and $[\tau, T]$, we obtain a  piecewise continuous solution  of \eqref{pbl:ntl} on $[0,T]$. Clearly this procedure can be repeated if the source measure $\sigma_0$ contains a finite number of atoms. The  resulting solution  of \eqref{pbl:ntl} is piecewise Lipschitz continuous  on a finite number of disjoint intervals in $[0,T]$.
\end{remark}

\section{A multiscale model for traffic flow on networks}\label{trafficmodel}
In this section we describe a nonlocal velocity $v[m]$ suitable to describe and predict the evolution of traffic flow on a road network.

There exists a wide literature on nonlocal fluxes for vehicular and pedestrian traffic: for example in \cite{Colombo_Garavello} a nonlocal term is used to modify the direction of motion of pedestrians and in \cite{Crippa} to describe interactions among different populations. An interesting possibility for vehicular traffic is proposed in \cite{Goatin_Rossi,Goatin_Scialanga}, where nonlocal terms are used as parameters to select the right flux in classical hyperbolic models.

Taking inspiration from similar models describing collective dynamics of crowds, see \cite{cristiani2014BOOK}, we consider a positive velocity fields given by
\begin{equation}\label{velocity}
 v[m](x) := v_{f}(x) - v_{i}[m](x)
\end{equation}
where $v_{f}: \G\to\R^+$ is the desired velocity, or free flow speed, representing the speed of a vehicle in a free road, and $v_{i}[m](x)$ is the interaction speed due to the presence of a vehicle distribution $m\in\cM^+(\G)$ on the network $\G$. Our aim is to identify an  appropriate expression of $v[m](x)$ consistent with the traffic flow model and satisfying the hypotheses $\textbf{(H1)--(H3)}$.

Concerning the free flow speed $ v_{f}(x)$, which depends only on the state variable $x$, we assume that this function is positive, bounded and Lipschitz continuous on each arc  of the network $\G$. Hence $\textbf{(H1)--(H3)}$ are easily verified.

Since for every $x\in\G$ the interaction velocity is a function $v_i[\cdot](x):\cM^+(\G)\to\R^+$, it is natural to define $v_i$ as the functional
$$ v_i[m](x) := \int_{\G}K(x,y)dm(y), $$
with interaction kernel $K\in BV(\G\times\G)$. If $K$ is nonnegative and bounded by a positive constant $C$, then for every $x\in\G$ it results
$$ |v_i[m](x)|\leq C m(\G), $$
hence $\textbf{(H1)}$ is satisfied.

Given $e_j\in\cE$ and $x\in e_j$, for every $m,\,\mu\in\cM^+(\G)$ we have, by the boundedness of $K$,
$$ |v[m](x)-v[\mu](x)|=|\int_{\G}K(x,y)d(m-\mu)(y)|\leq C\|m-\nu\|_{BL}^\ast $$
therefore also $\textbf{(H2)}$ holds.

The Lipschitz continuity with respect to $x$ is more delicate. In \cite[Section 5]{cristiani2014BOOK}, with reference to the whole Euclidean space $\R^d$, the authors consider a kernel for the interaction velocity of the form $K(x,y)=k(|x-y|)\chi_{\mathcal D(x)}(y)$, where $k:\R^+\to\R$ is a Lipschitz continuous non-increasing function, $\chi_{\mathcal{D}(x)}$ is the characteristic function of the set $\mathcal{D}(x)$
and $\mathcal D(x)$ is the visual field of a car driver at $x$ defined as
$$ \mathcal D(x):=\{y\in\R^d \text{ such that } |x-y|\leq R\} $$
for a given  visual radius $R>0$.

In order to generalize this approach to the case of networks we consider an interaction kernel of the form
$$ K(x,y)=k(d_\G (x,y))\eta_{x}(y), $$
where $k$, again Lipschitz continuous and non-increasing, represents the interactions among the vehicles on the network as a function of the distance. The crucial point is to properly define the function $\eta_{x}(y)$ representing the visual field of the drivers. We assume that a driver has a knowledge only of the distribution of the vehicles on the  roads adjacent to his/her current position and, on the basis of this information, he/she gives a certain priority to a possible route. After defining the visual field as
\[
\mathcal D(x)= \{y\geq x \text{ such that } d_\G(x,y)\leq R\},
\]
where $y\geq x$ is meant with respect to the orientation of the network, we assume that
\[R\le \min_{e_j\in \cE} \mathcal{L}(e_j) .\]
Hence,  given  $x\in e_k$, if  $x_i=\pi_k(L_k)\in\cV$  we have  $\mathcal D(x)\subset e_k\cup(\bigcup_{e_j\in\text{Out}(x_i)}e_j)$. Moreover,  for any $e_j\in\text{Out}(x_i)$, a weight $\alpha_{kj}$  is prescribed with the properties
\begin{align*}
	& 0\leq\alpha_{kj}\leq 1, \quad \sum_{j=1}^{J}\alpha_{kj}=1, \\
	& \alpha_{kj}=0 \quad \text{if either } e_k\cap e_j=\emptyset \text{ or } e_j\to e_k.
\end{align*}
We point out that the difference between the coefficients $p_{kj}(t)$ in \eqref{eq:pijk.sum.1} and $\alpha_{kj}$ is that the former represents the capacity of the junction $e_k\cap  e_j$ to allocate the traffic distribution while the latter translates the preference assigned to a given route by the drivers depending on the observed traffic distribution.
With the previous definitions, we consider an interaction velocity at $x\in e_k$ given by
$$ v_i[m](x)=\sum_{e_j\in\cE}\alpha_{kj}\int_{\G}k(d_{\G}(x,y))\chi_{\mathcal D(x)\cap (e_k\cup e_j)}(y)dm(y). $$
Note that  the support of $\chi_{\mathcal D(x)\cap (e_k\cup e_j)}$  is given by
$$ \{y\in e_k\cup e_j \text{ such that } x\to y,\,d_{\G}(x,y)\leq R\}, $$
which is isomorphic to a classical visual field for each $e_j$ for which $\alpha_{kj}\neq 0$.

To prove the Lipschitz continuity of $v_i[m]$ in the $x$-variable it is enough to check this property for the term
$$ \int_{\G}k(d_{\G}(x,y))\chi_{\mathcal D(x)\cap (e_k\cup e_j)}(y)dm(y). $$
Without loss of generality, we work directly with a parametrization of $e_k \cup e_j$ and we assume that $e_k$ is parametrized as $[0, L_k]$ and $e_j$ as $[L_k, L_k + L_j]$. In these terms,
$$ \mathcal D(x)\cap (e_k\cup e_j)=\{y\in[0,L_k+R]\subset [0,L_k+ L_j]\,:\,x\to y,\,|x-y|\leq R\}=:\mathcal{A}(x). $$
Taken $x_1,x_2\in [0, L_k]$ with $x_2 \to x_1$ and denoted $h:=|x_2 - x_1|$, we observe that $\mathcal{A}(x_2)=\mathcal{A}(x_1)+h$; then
$$ \chi_{\mathcal{A}(x_2)}(y)=
	\begin{cases}
		1 & \text{if } y-h\in\mathcal{A}(x_1) \\
		0 & \text{otherwise,}
	\end{cases}
$$
hence
$$ \chi_{\mathcal{A}(x_2)}(y)=\chi_{\mathcal{A}(x_1)}(y-h)=(\chi_{\mathcal{A}(x_1)}\circ\tau_{-h})(y) $$
where $\tau_{-h}$ is the translation by $-h$. It follows
\begin{align*}
	& \int k(|x_1-y|)\chi_{\mathcal{A}(x_1)}(y)dm(y)-\int k(|x_2-y|)\chi_{\mathcal{A}(x_2)}(y)dm(y) \\
	&= \int k(|x_1-y|)\chi_{\mathcal{A}(x_1)}(y)dm(y)-\int k(|x_1-(y-h)|)\chi_{\mathcal{A}(x_1)}(y-h)dm(y) \\
	&= \int\left[k(|x_1-y|)\chi_{\mathcal{A}(x_1)}(y)-(k(|x_1-\cdot|)\chi_{\mathcal{A}(x_1)}(\cdot))\circ\tau_{-h}(y)\right]dm(y) \\
	&= \int k(|x_1-y|)\chi_{\mathcal{A}(x_1)}(y)d(m-\tau_{-h}\#m)(y) \\
	&\leq \|m-\tau_{-h}\#m\|_{BL}^\ast=Kh=K|x_2-x_1|,
\end{align*}
whence the Lipschitz continuity with respect to $x$ as desired.

Notice that \eqref{velocity} does not guarantee the positivity of $v$. However, if we consider the velocity field
$$ v[m](x)=\max\{v_{f}(x)-v_{i}[m](x),0\} $$
then the boundedness and the Lipschitz continuity with respect to $x$ and $m$ are preserved and moreover $v$ is nonnegative.

\section*{Acknowledgements}
A.T. is member of GNFM (Gruppo Nazionale per la Fisica Matematica) of INdAM (Istituto Nazionale di Alta Matematica), Italy.

A.T. acknowledges that this work has been written within the activities of a starting grant funded by ``Compagnia di San Paolo'' (Turin, Italy).

\appendix
\section{Proofs of theorems of section~\ref{sec:existence}}\label{appendix}

\begin{proof}[Proof of Proposition \ref{lemma1}]
Let $t \in [0,T]$ and $n \in \{0,\hdots, 2^{N - 1}\}$ such that $t \in I_n^N$.
Then, by the representation formula \eqref{globchar}, we write
\[
\dual {m^N_t}{f} = \int_\G \sum_{\gamma \in \mathcal{A}(x)}f(\Phi^\gamma_t(x,t_n),t)p_{\gamma}(x) dm^N_{t_n}(x) + \sum_{x_i \in \cS} \int_{[0,T]}\sum_{\gamma \in \mathcal{A}(x_i)}f(\Phi^{\gamma}_{t}(x_i, s),t)p_{\gamma}(x_i)d\sigma_0^i(s).
\]
Hence, for every $f \in BL(\G\times[0,T])$ such that $\|f\|_{BL} \leq 1$ it follows
\begin{align*}
|\dual{m^N_t}{f}|&\leq \int_{\G} \|f\|_{BL} \left(\sum_{\gamma \in \mathcal{A}(x)}p_{\gamma}(x)\right) dm^N_{t_n}(x) + \sum_{x_i \in \cS}\int_{[0,T]}\|f\|_{BL} \left(\sum_{\gamma \in \mathcal{A}(x_i)}p_{\gamma}(x_i)\right)d\sigma_0^i(s)\\
&\leq\|m^N_{t_n}\|_{BL}^\ast + \sum_{x_i \in \cS}\|(\sigma_0^i)\llcorner{(t_{n}, t]}\|_{BL}^\ast=\|m^N_{t_n}\|_{BL}^\ast + \|(\sigma_{0})\llcorner{(t_n, t]}\|_{BL}^\ast,
\end{align*}
where we have used the property $\sum_{\gamma \in \mathcal{A}(x)}p_{\gamma}(x) = 1$ for all $x \in \G.$ Taking the supremum over $f \in BL(\G\times[0,T])$ we get
\[
\|m^N_t\|_{BL}^\ast\leq \|m_{t_n}\|_{BL}^\ast + \|(\sigma_{0})\llcorner{(t_n, t]}\|_{BL}^\ast;
\]
Applying the previous inequality recursively for $m\in  \{0,\hdots, n\}$ we get \eqref{eqnlemma1}.



We now prove \eqref{lip}. Let $N\in\N$ and $s,t \in [0,T]$ such that $s<t$ with $s\in I^N_{n}$, $t\in I^N_{k}$ for $n\neq k$ in $\{0,\dots,2^N-1\}$. This means
$$ t_{n} \leq s < t_{n+1}  < \hdots < t_{k}  \leq t \leq t_{k+1}. $$
Clearly
$$ m^N_t-m^N_s=\left(m^N_t-m^N_{t_{k}}\right)+(m^N_{t_{n+1}}-m^N_s)+\sum_{l=n+1}^{k}(m^N_{t_{l+1}}-m^N_{t_{l}}), $$
which implies
\begin{equation}\label{proofprop1}
\|m^N_t - m^N_s\|_{BL}^\ast \leq \|m^N_{t} - m^N_{t_{k}}\|_{BL}^\ast + \|m^N_{t_{n+1}} - m^N_{s}\|_{BL}^\ast + \sum_{l=n+1}^{m}\|m^N_{t_{l+1}} - m^N_{t_{l}}\|_{BL}^\ast.
\end{equation}
Therefore we need to estimate $\|m^N_{t_{l+1}} - m^N_{t_{l}}\|_{BL}^\ast$. Let $f \in BL(\G\times[0,T])$ such that $\|f\|_{BL}^\ast \leq 1.$ Then, for every $t \in I^N_n$,
\begin{align*}
	|\dual{m^N_t-m^N_{t_n}}{f}| &\leq \int_{\G}\sum_{\gamma\in\mathcal{A}(x)}|f(\Phi^{\gamma}_t(x,t_n),t)-f(x,t_n)|dm^N_{t_n}(x) \\
	&\phantom{\leq} +\left|\sum_{x_i\in\cS}\int_{(t_n, t]}f(\Phi_t(x_i, s),t)d\sigma_0^i(s)\right| \\
	&\leq \int_{\G}\left(\sum_{\gamma\in\mathcal{A}(x)}p_{\gamma}(x)\right)(d(\Phi_t^\gamma(x,t_n),x)+|t-t_n|)dm_{t_n}^N(x)
		+\|(\sigma_0)\llcorner{(t_n, t]}\|_{BL}^\ast.
\end{align*}
By definition of $\Phi^{\gamma}$, it  follows $d(\Phi^\gamma_t(x,t_n), x) \leq \int_{t_n}^t v_n^N (\Phi^\gamma_s(x, t_n))ds \leq |t - t_n| V_{max}.$ 
Then, applying  \eqref{eqnlemma1} and taking the supremum over $f \in BL(\G\times[0,T])$ such that $\|f\|_{BL}\leq 1$, we can write
\begin{equation}\label{ineq1}
\|m^N_t - m^N_{t_n}\|_{BL}^\ast \leq C| t - t_n | + \sigma_0((t_n, t]),
\end{equation}
where $C=(1+V_{max})(\|m_0\|_{BL}^\ast + \|\sigma_0\|_{BL}^\ast)>0.$
Using \eqref{ineq1}
in \eqref{proofprop1}, we get \eqref{lip}.
\end{proof}



\begin{proof}[Proof of Proposition \ref{momentact}]
For simplicity, in this proof we denote $v[m^N_{t^N_n}]$ by $v_n^{N}$ and $d(\cdot ):=d_{\G}(\cdot,V)$ for a fixed $V\in\cV$. Let $N\in\N$, $t\in [0,T]$ and $n\in\{0,\dots,2^N-1\}$ such that $t \in I^N_{n-1}$. By Lemma \ref{lemmoment}, $m^N_t \in \cM^+(\G)$ has finite $p$-moment over $\G$ if and only if it has finite $p$-moment on every arc $e_j\in\cE$ such that $\mathcal{L}(e_j)=+\infty$.

We consider first the case $p=1$. If $e_j \in\text{Inc}(V)\cup\text{Out}(V)\subset\cE$, such that  $\mathcal{L}(e_j) = +\infty$, there are two possibilities:
\begin{enumerate}
\item[i)] $e_j\in\text{Inc}(V)$;
\item[ii)] $e_j\in\text{Out}(V)$.
\end{enumerate}
If i) occurs, we  parametrize $e_j \in \cE$ as $(-\infty; 0]$. For every $t \in I^N_{n-1}$, we denote by $\Phi^{e_j}_t$ the flow map in $e_j$ produced by the velocity $v_{n-1}^{N}$. By the definition in \eqref{tau}, we have
$$ \tau_{n,N}(t) = \inf\{t\geq t_{n-1}^N: \Phi^{e_j}_t(x,t_{n-1}^N) = \pi_j(0)\}. $$
Then the first moment in $e_j$ of $m^N_t$ can be estimated as
\begin{align*}
	\int_{(-\infty; 0]}|x|dm^{N,j}_t(x) &= \int_{(-\infty; \tau^{-1}_{n,N}(t)]}|\Phi^{e_j}_{t}(x,t_{n - 1})|dm^{N,j}_{t_{n-1}}(x) \\
	&\leq \int_{(-\infty; \tau^{-1}_{n,N}(t)]}|x|dm^{N,j}_{t_{n-1}}(x) + m^N_{t_{n-1}}((-\infty; \tau^{-1}_{n,N}(t)])V_{max}\Delta t^N \\
	&\leq \int_{(-\infty; 0]}|x|dm^{N,j}_{t_{n-1}} + m^{N,j}_{t_{n-1}}(e_j)\Delta t^N V_{max}.
\end{align*}
Applying iteratively the previous argument for $k\in \{0,1,\hdots, n-1\}$, we get
\begin{align*}
	\int_{e}|x|dm^N_t(x) &\leq \int_{e}|x|dm^j_0(x) + V_{max}\Delta t^N \sum_{k=0}^{n-1}m^N_{t_{k}}(e_j) \\
	&\leq \int_{e}|x|dm^j_0(x) + V_{max}\frac{T}{2^N}\sum_{k=0}^{2^N-1}m^{N}_{t_k}(e_j).
\end{align*}
By Lemma \ref{lemma1} we have
\begin{equation}\label{pezzo0}
	\int_{e}|x|dm^N_t(x)  \leq \int_{e}|x|dm^j_0(x) + V_{max} T C.
\end{equation}
For the measure $m^{N,j}_{x=V}\in\cM^+([0,T])$, namely the projection of $m^{N,j}$ at $x_i$, by \eqref{outnu} we estimate
\begin{equation}\label{pezzo1}
\|m_{x=V}^{N,j}\|_{BL}^\ast = m_{x=V}^{N,j}([0,T]) \leq m_0^j ((-V_{max}T,0]) \leq m^j_0(e_j).
\end{equation}

If ii) occurs, we have a similar proof. Indeed, thanks to the characterisation \eqref{outnu}, we can write
$$ \int_{[0,+\infty)}|x|dm^{N,j}_t(x)=\int_{[0,+\infty)}|\Phi^{e_j}_{t}(x,t_{j-1})|dm^{N,j}_{t_{n-1}}(x)
	+\int_{[\varsigma^{-1}_j(t),t]}|\Phi^{e_j}_{t}(0,s)|dm^{N,j}_{x=V}(s). $$
The first integral at the right-hand side can be estimated like in \eqref{pezzo0}, while for the second one we have
\begin{align*}
	\int_{[\sigma^{-1}_e(t),t]}|\Phi^e_{t}(0,s)|dm_{x=x_i}^j(s) &\leq V_{max}\int_{[\sigma^{-1}_e(t),t]}|t-s|dm_{x=V}^j(s) \\
	&\leq V_{max}\Delta t^N m_{x=V}^j( I_{n-1}^N),
\end{align*}
which, thanks to \eqref{pezzo1} and Theorem \ref{existencenet}, is finite and bounded by a constant which depends on $\|m_0\|_{BL}^\ast$, $\|\sigma_{0}\|_{BL}^\ast$, $V_{max}$ and $T$.

To conclude the proof, we need to show an analogous statement for $p=2$. However, we can observe that
\begin{align*}
	|\Phi^{e_j}_{t}(x,t_{n-1})|^2 &= |x|^2+\left|\int_{t_{j-1}}^t v_{j-1}^{N}(\Phi_s^e(x,t_{j-1}))\right|^2
		+ 2|x|\cdot\left|\int_{t_{n-1}}^t v_{n-1}^{N}(\Phi_s^{e_j}(x,t_{n-1}))\right| \\
	&\leq |x|^2+(V_{max}\Delta t^N)^2+2V_{max}\Delta t^N|x|,
\end{align*}
and, for $s \in [\varsigma^{-1}_j(t),t]$,
$$ |\Phi_t^{e_j}(0,s)|^2=\left|\int_{s}^t v_{j-1}^{N}(\Phi_u^{e_j}(0,s))du\right|^2\leq (V_{max}\Delta t^N)^2. $$
Then, taking advantage of the uniform bound on the first moment, we repeat the argument used for $p=1$ to obtain the uniform bound also on the second moment.
\end{proof}

\bibliographystyle{amsplain}
\bibliography{references}
\end{document}